\documentclass[12pt,a4paper]{article}
\usepackage{amsthm}
\usepackage{color}
\usepackage{amsmath}
\usepackage{amssymb}
\usepackage{epsfig}
\usepackage[english]{babel}

\theoremstyle{plain}
\newtheorem{thm}{Theorem}
\newtheorem{prp}[thm]{Proposition}
\newtheorem{lemma}[thm]{Lemma}

\newtheorem{defin}{Definition}
\newtheorem{remark}[thm]{Remark}
\newcommand{\po}{{\mathbb P}}

\begin{document}
\title{Characterization of the law for 3D stochastic hyperviscous fluids}
\author{Benedetta Ferrario\footnote{Universit\`a di Pavia,
    Dipartimento di Matematica ''F. Casorati'', via
  Ferrata 5, 27100 Pavia, Italy. Email:
benedetta.ferrario@unipv.it}}
\date{\today}
\maketitle

\begin{abstract}
We consider the 3D hyperviscous Navier-Stokes equations in 
vorticity form, where the dissipative term $-\Delta
\vec \xi$ of the Navier-Stokes equations is substituted by $(-\Delta)^{1+c}
\vec \xi $. 
We investigate how big the  correction term $c$ has to be
in order to prove, by means of Girsanov transform, that the 
vorticity equations are equivalent (in law) to 
easier reference equations obtained by neglecting the stretching
term. 
This holds as soon as $c>\frac 12$, improving 
previous results obtained with
$c>\frac 32$ in a different setting in \cite{F,Kos}.
\end{abstract}
\noindent
{\bf MSC2010}: 76M35, 60H15, 35Q30.\\
{\bf Keywords}: Hyperviscous fluids, well-posedness, Girsanov formula.

\section{Introduction}
The stochastic Navier-Stokes equations, governing the motion of a
homogeneous and incompressible viscous fluid, are 
\begin{equation}\label{vel}
\begin{cases}
\dfrac{\partial \vec v}{\partial t}-\nu \Delta \vec v
  +(\vec v\cdot \nabla) \vec v+\nabla p  =\vec f +\vec n
\\ \nabla \cdot \vec v=0
\end{cases}
\end{equation}
where the unknown are the  velocity $\vec v$ and  the pressure $p$;
the data are the viscosity $\nu>0$,  the deterministic forcing
term $\vec f$ and the random one $\vec n$.

Working in a bounded three dimensional 
spatial domain with suitable boundary conditions,
it is known that for initial velocity of finite energy and suitable
forcing terms there exists a
weak solution  to \eqref{vel} defined for any positive time, 
but uniqueness is an open problem. On the other side, 
more regular initial velocities provide  existence and uniqueness
of a solution, which is only  local in time.
For these results we refer to \cite{temam}
for the deterministic equations (the case $\vec n=\vec 0$) and to \cite{cime}
for the stochastic ones (the case $\vec n\neq \vec 0$).

However, suitable modifications of the first equation in \eqref{vel} provide
better results.
Let us consider 
the hyperviscous model 
\begin{equation}\label{NS}
\begin{cases}
\dfrac{\partial \vec v}{\partial t}+\nu (-\Delta)^{1+ c } \vec v
  +(\vec v\cdot \nabla) \vec v+\nabla p = \vec f + \vec n 
\\ \nabla \cdot \vec v=0
\end{cases}
\end{equation}
We consider $c>0$, whereas it reduces to the Navier-Stokes system for $c=0$.
This model is widely used in computer simulations (see e.g. \cite{Fr}, \cite{SMC} and references therein).
It turns out that large enough values of the parameter $c$
 provide better 
mathematical properties of system \eqref{NS}.

As far as the well posedness of \eqref{NS} is concerned, 
the condition $c\ge \frac 14$ allows to prove that  there exists a unique
global solution for the hyperviscous  Navier-Stokes equations
\eqref{NS}.
This is based on the fact that  the operator
$(-\Delta)^{1+c}$ has a more regularizing effect than the
Laplacian itself and $c\ge \frac 14$ provides a sufficient regularity
to prove uniqueness of the global weak solution.
The result has been proved first for
 integer values of $ c \ge 1$, 
  both in the stochastic (see \cite{sri}) and deterministic case
(see \cite{lady}).
Then, these results have been improved allowing $ c $ to be non integer
(see \cite{fe-proc} for
the stochastic case and \cite{MS} for the deterministic one).

A further question concerns the characterization of the law of the
process solving \eqref{NS} with a stochastic force.
When $\vec f=\vec 0$ and $\vec n$ 
is a Gaussian random field, white in time
and coloured in space, 
Gallavotti (see \cite{Gal}, Ch 6.1) suggested 
to use Girsanov transform to relate the law of the stochastic Navier-Stokes
equations
with that of the stochastic Stokes equations, which are linear equations
obtained from the Navier-Stokes ones by neglecting
the non linear term $(\vec v\cdot \nabla) \vec v$. The formula given in
\cite{Gal} when $c=0$ is formal, but this idea can be used also for
the hyperviscous fluids. Actually, a rigorous result  has been
proved  in \cite{Kos}, \cite{F}:  for $c>\frac
32$  the law of the process $\vec v$ solving 
\begin{equation}\label{sis-nonlin}
\begin{cases}
\dfrac{\partial \vec v}{\partial t}+\nu (-\Delta)^{1+ c } \vec v
  +(\vec v\cdot \nabla) \vec v+\nabla p =  \vec n 
\\ \nabla \cdot \vec v=0
\end{cases}
\end{equation}
is equivalent to
the law of the process $\vec z$ solving the stochastic
hyperviscous Stokes  system
\begin{equation}\label{sis-lin}
\begin{cases}
\dfrac{\partial \vec z}{\partial t}+\nu (-\Delta)^{1+ c } \vec z
  +\nabla p  =\vec n
\\ \nabla \cdot \vec z=0
\end{cases}
\end{equation}
This holds in the 2D and  in the 3D setting and implies that all
what holds a.s. for the hyperviscous Stokes problem \eqref{sis-lin} holds a.s. 
for the hyperviscous Navier-Stokes problem \eqref{sis-nonlin} as well. In other words: 
the advection term $(\vec v\cdot \nabla) \vec v$ takes second place 
to the dissipative term $(-\Delta)^{1+ c } \vec v$ for $c$ large enough.
This means that  hyperviscosity with $c>\frac 32$ 
 changes drastically the nature of the 
equations of motion of the fluid. 
This remark already appeared  in \cite{Fr}, where the authors 
discuss {\it artifacts} arising in numerical simulation of hyperviscous
fluids.
The mathematical  representation of the law of
$\vec v$ by means of Girsanov transform, which reduces the analysis of the
law of $\vec v$ to the analysis of the law of the linear problem for
$\vec z$, gives evidence in support of the fact that hyperviscous fluid models  with $c>\frac 32$ are far
away from the real turbulent fluids.

But, what happens for smaller values
of the correction term, i.e. for $c\le \frac 32$? 
To answer this question,
we change the auxiliary process.
First of all we write the Navier-Stokes system in
vorticity form
\begin{equation}\label{vorti}
\begin{cases}
\dfrac{\partial \vec \xi}{\partial t}+\nu (-\Delta)^{1+ c }\vec \xi 
   + (\vec v\cdot \nabla)\vec \xi -  (\vec \xi \cdot \nabla) \vec v
= \nabla \times \vec  n\\
\nabla \cdot \vec v=0\\
\vec \xi= \nabla \times \vec v
\end{cases}
\end{equation}
Notice that  the first equation can be rewritten as
\[
\dfrac{\partial \vec \xi}{\partial t}+\nu (-\Delta)^{1+ c }\vec \xi 
   + P[(\vec v\cdot \nabla)\vec \xi] -  P[(\vec \xi \cdot \nabla) \vec v]
= \nabla \times \vec  n
\]
where $P$ is the projection operator onto the space of divergence free
vector fields (see details in Section \ref{sec:ma}).

The idea is to simplify the vorticity equation by neglecting only the 
vorticity stretching term, getting
\begin{equation}\label{sis-eta}
\begin{cases}
\dfrac{\partial \vec \eta}{\partial t}+\nu (-\Delta)^{1+ c }\vec \eta 
   + P[(\vec v\cdot \nabla)\vec \eta]
= \nabla \times \vec  n\\
\nabla \cdot \vec v=0\\
\vec \eta= \nabla \times \vec v
\end{cases}
\end{equation}
This system has the same structure as the 2D vorticity system, but we consider it in the 3D setting.
Indeed, in the 2D setting the vorticity is a vector 
orthogonal to the plane where the fluid moves and 
therefore the  term $(\vec \xi \cdot \nabla) \vec v$ vanishes.
Therefore, systems \eqref{vorti} and \eqref{sis-eta}  
are different only in the 3D setting. Let us compare them.

From the mathematical point of view we shall prove that 
system \eqref{sis-eta} is well posed for any $c\ge 0$, whereas 
the well posedness of the full system \eqref{vorti}
has been proved by assuming $c\ge \frac 14$.

On the other hand, 
the vorticity stretching term
$(\vec \xi\cdot \nabla)\vec v$ 
 is essential in 3D fluids (see e.g. \cite{Frisch} Ch 9);
it is responsible of the peculiar features of
3D turbulence, which is very different from and more involved than 2D
turbulence.
Thus one expects the dynamics of 
\[
\begin{cases}
\dfrac{\partial \vec \xi}{\partial t}-\nu\Delta\vec \xi 
   + P[(\vec v\cdot \nabla)\vec \xi -  (\vec \xi \cdot \nabla) \vec v]
= \nabla \times \vec  n\\
\nabla \cdot \vec v=0\\
\vec \xi= \nabla \times \vec v
\end{cases}
\]
to be very different from that of
\[
\begin{cases}
\dfrac{\partial \vec \eta}{\partial t}-\nu \Delta\vec \eta 
   + P[(\vec v\cdot \nabla)\vec \eta]
= \nabla \times \vec  n\\
\nabla \cdot \vec v=0\\
\vec \eta= \nabla \times \vec v
\end{cases}
\]

Now, the question is:
what happens if we introduce  hyperviscosity 
$(-\Delta)^{1+c}$?
Our main theorem states the
equivalence of laws of the solution processes of systems \eqref{vorti}
and \eqref{sis-eta}
under the assumption $ c  > \frac{1}{2}$. 
Again our result gives evidence that the hyperviscous models with
$c>\frac 12$ do not represent well the real 3D turbulence, since the
effect of the vorticity stretching term are not relevant when $c>\frac
12$.

Finally, we present this paper. In the next section we define the
functional spaces and the noise term.  Section \ref{sec-stime}
presents various technical results. Then we start to analyze the 
 main equations: the linear problem in Section \ref{sec-ou}, the
 auxiliary problem \eqref{sis-eta} in Section \ref{sec-eta} and the full
 vorticity problem \eqref{vorti} in Section \ref{sec-vor}. 
The main result on the equivalence of the laws is proved in Section \ref{s:equiv}.

\section{Mathematical setting}\label{sec:ma}
We denote  a 3D vector as $\vec k=(k^{(1)},k^{(2)},k^{(3)})$; 
we define $\mathbb Z^3_0=\mathbb Z^3\setminus \{\vec
0\}$ and $\mathbb Z^3_+=
\{k^{(1)}>0\}\cup \{k^{(1)}=0,k^{(2)}>0\}\cup\{k^{(1)}=0,k^{(2)}=0,k^{(3)}>0\}$. 
Then for any $\vec k \in \mathbb Z^3_0$, there exist two unit 
vectors $\vec b_{\vec k,1}$ and $\vec b_{\vec k,2}$, 
orthogonal to each other and belonging to
the plane orthogonal to $\vec k$; we choose these vectors in such a
way that $(\vec b_{\vec k,1}, \vec b_{\vec k,2}, \frac {\vec k}{|\vec k|})$ 
is a right-handed orthonormal frame and 
$\vec b_{\vec k,j}=-\vec b_{-\vec k,j}$.

We work on the 3D torus, that is we deal with functions defined on
$\mathbb R^3$ and $[-\pi,\pi]^3$-periodic.
We set $D=[-\pi,\pi]^3$.
As usual, in the periodic case we assume that the mean value of the
vectors we are dealing with is zero. 
This gives a simplification in the mathematical treatment, but it does
not prevent to consider non zero  mean value vectors. Actually, if we
can analyse the problem for zero mean vectors then the problem without
this assumption can be dealt with in a similar way
(see \cite{temam2}).

The velocity
vector $\vec v$ is divergence free by assumption and the vorticity
vector $\vec \xi$ is divergence free by construction. We can
write any zero mean, periodic, divergence
free vector $\vec u$ in Fourier series as 
\[
\vec u(\vec x)=
\sum_{\vec k\in \mathbb Z^3_0} [u_{\vec k,1}\vec b_{\vec k,1}
   +u_{\vec k,2}\vec b_{\vec k,2}] e^{i \vec k \cdot \vec x}, \qquad
\vec x \in \mathbb R^3
\]
where $u_{\vec k,1},u_{\vec k,2} \in \mathbb C$, 
with the condition  $\overline u_{\vec k,j}=-u_{-\vec k,j}$
in order to have a real vector $\vec u(\vec x)$. 

When needed, we use the notation $\vec v$ and $\vec \xi$ to make
precise that we deal with the velocity or vorticity vector. For
instance, we have $\vec \xi=\nabla \times \vec v$, but we can also
 express the velocity in
terms of the vorticity, solving
\begin{equation}\label{vel-vort}
\begin{cases}
-\Delta \vec v=\nabla \times \vec \xi\\
\nabla\cdot \vec v=0\\
\vec v \text{ periodic}
\end{cases}
\end{equation}
More explicitly
\begin{multline}\label{xi-e-v}
\vec \xi(\vec x)= \sum_{\vec k \in \mathbb Z^3_0}
(\xi_{\vec k,1}\vec b_{\vec k,1}+\xi_{\vec k,2}\vec b_{\vec k,2})e^{i
  \vec k \cdot \vec x}
\\\Longrightarrow\quad
\vec v(\vec x)=i\sum_{\vec k \in \mathbb Z^3_0}\frac 1{|\vec k|}
(\xi_{\vec k,1}\vec b_{\vec k,2}-\xi_{\vec k,2}\vec b_{\vec k,1})e^{i
  \vec k \cdot \vec x}
\end{multline}

We now define the functional spaces.
Let $L_2$ denote the subspace of $[L^2(D )]^3$ consisting of 
zero mean, periodic, divergence
free vectors (this condition has to be understood in the
distributional sense):
\[
L_2=\Big\{\vec u(\vec x)=
\sum_{\vec k\in \mathbb Z^3_0} [u_{\vec k,1}\vec b_{\vec k,1}
   +u_{\vec k,2}\vec b_{\vec k,2}] e^{i \vec k \cdot \vec x}:
  \sum_{\vec k\in \mathbb Z^3_0} 
(|u_{\vec k,1}|^2+|u_{\vec k,2}|^2)<\infty\Big\}
\]
This is a Hilbert space with scalar product
\[
\langle \vec u, \vec v \rangle= (2\pi)^3 
 \sum_{\vec k\in \mathbb Z^3_0} 
(u_{\vec k,1} \overline v_{\vec k,1}+u_{\vec k,2} \overline v_{\vec k,2})
\]
The space $L_2$ is a closed subspace of $[L^2(D )]^3$; 
we decide to put the subindex in $L_2$  in order to distinguish them.

Moreover, for any integer $n$  we define the projection operator 
$\Pi_n$ as a linear bounded operator in $L_2$ such that
\[
\Pi_n \left(\sum_{\vec k\in \mathbb Z^3_0} [u_{\vec k,1}\vec b_{\vec k,1}
   +u_{\vec k,2}\vec b_{\vec k,2}] e^{i \vec k \cdot \vec x}\right)
= \sum_{0<|\vec k|\le n} [u_{\vec k,1}\vec b_{\vec k,1}
   +u_{\vec k,2}\vec b_{\vec k,2}] e^{i \vec k \cdot \vec x}
\]
and we set $H_n=\Pi_n L_2$.

For  $p>2$ we define the Banach spaces 
\[
 L_p= L_2\cap  [L^p(D )]^3
\]
These are Banach spaces with norms inherited from  $[L^p(D )]^3$.\\
We denote by $P$ the projection operator from $[L^p(D )]^3$ onto
$L_p$. We have that $P[(\vec v\cdot \nabla)\vec \xi-
(\vec \xi\cdot \nabla)\vec v]=0$.
Indeed, the vorticity transport term $(\vec v\cdot \nabla)\vec \xi$ and the 
vorticity stretching term $(\vec \xi \cdot \nabla) \vec v$ 
are not divergence free vector fields; so
$P[(\vec v\cdot \nabla)\vec \xi]\neq (\vec v\cdot \nabla)\vec \xi$ and 
$P[(\vec \xi \cdot \nabla) \vec v]\neq (\vec \xi \cdot \nabla) \vec v$.
However, their difference is divergence free, being given by the curl form
$\nabla \times [(\vec v \cdot \nabla )\vec v]$.
Moreover, if $\vec\phi$ is a divergence free vector field (i.e. $P\vec \phi=\vec \phi$), then
\[
\langle P[(\vec \xi \cdot \nabla) \vec v],\vec \phi\rangle 
=
 \langle (\vec \xi \cdot \nabla) \vec v,\vec \phi\rangle 
\]

For any $a\in \mathbb R$ we define the fractional 
powers of the Laplace operator; 
formally, if
\[
\vec u(\vec x)=
\sum_{\vec k\in \mathbb Z^3_0} [u_{\vec k,1}\vec b_{\vec k,1}
   +u_{\vec k,2}\vec b_{\vec k,2}] e^{i \vec k \cdot \vec x}
\]
then
\[
 (-\Delta)^a \vec u(\vec x) =\sum_{\vec k\in \mathbb Z^3_0} 
|\vec k|^{2a}[u_{\vec k,1}\vec b_{\vec k,1}
   +u_{\vec k,2}\vec b_{\vec k,2}] e^{i \vec k \cdot \vec x}
\]
Thus, for $b\in \mathbb R$  we define the Hilbert spaces
\[
 H^b=\{\vec u (\vec x)=
\sum_{\vec k\in \mathbb Z^3_0} [u_{\vec k,1}\vec b_{\vec k,1}
   +u_{\vec k,2}\vec b_{\vec k,2}] e^{i \vec k \cdot \vec x}: 
\sum_{\vec k\in \mathbb Z^3_0} |\vec k|^{2b} 
(|u_{\vec k,1}|^2+|u_{\vec k,2}|^2) <\infty\}
\]
with scalar product 
\[
\langle \vec u, \vec v \rangle_b= (2\pi)^3 
 \sum_{\vec k\in \mathbb Z^3_0} |\vec k|^{2b}
(u_{\vec k,1} \overline v_{\vec k,1}+u_{\vec k,2} \overline v_{\vec
   k,2})\equiv 
\langle (-\Delta)^{\frac b2}\vec u, (-\Delta)^{\frac b2}\vec v \rangle
\]
The duality between $H^b$ and $H^{-b}$ 
(or between $[H^b(D)]^3$ and $[H^{-b}(D)]^3$)
is again denoted by $\langle\cdot, \cdot \rangle$.

For $b>0$ and  $p>2$, we define the generalized Sobolev spaces $H^b_p$ 
\[
 H^b_p=\{\vec u \in L_p: (-\Delta)^{\frac b2}\vec u \in L_p\}
\]
which are Banach spaces with norms
\[
 \|\vec u\|_{H^b_p}=\|(-\Delta)^{\frac b2}\vec u\|_{L_p}
\]
When $b \in \mathbb N$, $H^b_p$ are the  Sobolev spaces.
We recall the Sobolev embedding theorem (see \cite{Sob} Ch 1 \S 8)
\begin{itemize}
\item
if $1<p<q<\infty$ with $\frac 1q=\frac 1p-\frac{a-b}3$, then the following inclusion holds
\[
H^a_p\subset H^b_q
\]
and there exists a constant $C$ (depending on $a-b,p,q$) such that
\[
\|\vec v \|_{H^b_q}\le C \|\vec v\|_{H^a_p}
\]
\item
if $1<p<\infty$ with $ 3 <ap$, then the following inclusion holds
\[
H^a_p\subset L_\infty
\]
and there exists a constant $C$ (depending on $a,p$) such that
\[
\|\vec v \|_{L_\infty} \le C \|\vec v\|_{H^a_p}
\]

\end{itemize}

The Poincar\'e inequality holds, because of the zero mean value
assumption, and therefore $\|\vec u\|_{H^b_p}$ is equivalent to 
$( \|\vec u\|^p_{L_p}+\|\vec u\|^p_{H^b_p})^{1/p}$, which appears usually in
the definition of the generalized Sobolev spaces.

Moreover for $\vec \xi=\nabla \times \vec v$,  the norms 
$\|\vec v\|_{H^b_p}$ and $\|\vec \xi\|_{H^{b-1}_p}$
are equivalent (see \eqref{xi-e-v}).

For any $t>0$ and $b>0$, the linear operator $e^{-t(-\Delta)^b}$, formally
defined as
\[
 e^{-t(-\Delta)^b}\left(\sum_{\vec k\in \mathbb Z^3_0} [u_{\vec k,1}\vec b_{\vec k,1}
   +u_{\vec k,2}\vec b_{\vec k,2}] e^{i \vec k \cdot \vec x}\right) 
=
\sum_{\vec k\in \mathbb Z^3_0}e^{-t|\vec k|^{2b}} [u_{\vec k,1}\vec b_{\vec k,1}
   +u_{\vec k,2}\vec b_{\vec k,2}] e^{i \vec k \cdot \vec x}
\]
is a contraction operator in $L_p$ for any $p\ge 2$.

\medskip

Next, we define the random forcing term.
We consider a noise $d\vec n$ of the form 
$d(-\Delta)^{-b}\vec w$, where $\vec w$ 
is a cylindrical Wiener process in $L_2$ (see,
e.g., \cite{dpz}). 
We can represent it as follows.
Suppose we are given  a Brownian stochastic basis, i.e.
a probability space $\left(\Omega,\mathcal{F},\po\right)$ and 
a filtration $( \mathcal{F}_{t})_{t\geq0}$; we denote by $\mathbb E$
the mathematical expectation with respect to $\po$.
Let $\{\beta_{\vec k,1}, \beta_{\vec k,2}\}_{\vec k\in \mathbb Z^3_{+}}$ 
be a double sequence of complex valued independent
Brownian motions on $\left(\Omega,\mathcal{F},\left(\mathcal{F}_{t}\right)
_{t\geq0},\po\right)  $; namely, the sequence  
$\{\Re\beta_{\vec k,j},\Im\beta_{\vec k,j}\}_{\vec k \in \mathbb
  Z^3_{+};j=1,2}$
consists of real valued processes that are 
independent, adapted to $\left(  \mathcal{F}_{t}\right)  _{t\geq0}$, 
continuous for $t \ge 0$ and null at $t=0$, 
with increments  on  any time interval $[s,t]$
that are $N\left(  0,t-s \right)  $-distributed
and independent of $\mathcal{F}_{s}$.

Moreover, for $-\vec k \in \mathbb Z^3_{+}$ 
 let $\beta_{\vec k,j}=-\overline \beta_{-\vec k,j} $.
Then  
\begin{equation} \label{cilindrico}
 \vec w(t,\vec x)= \sum_{\vec k\in \mathbb Z^3_0} 
 [\vec b_{\vec k,1}\beta_{\vec k,1}(t) +\vec b_{\vec
     k,2}\beta_{\vec k,2}(t)]e^{i\vec k \cdot \vec x}
\end{equation} 
is a cylindrical Wiener process in $L_2$. Its paths do not live in the
space $C(\mathbb R_+;L_2)$; they are less regular in space. Indeed
\[
\mathbb E\|(-\Delta)^a \vec w(t)\|_{L_2}^2
=2 t \sum_{\vec k\in \mathbb Z^3_0} |\vec k|^{2a}
\]
which is finite if and only if $a<-\frac 32$.

Within this setting, we write  system \eqref{vorti} for the vorticity as
\begin{equation}\label{vort}
\begin{cases}
d\vec \xi+\left( (-\Delta)^{1+ c }\vec \xi 
   + P[(\vec v\cdot \nabla)\vec \xi] - P[(\vec \xi \cdot \nabla) \vec
     v]\ \right)
\ dt
= (-\Delta)^{-b} d\vec w\\
\nabla \cdot \vec v=0\\
\vec \xi= \nabla \times \vec v
\end{cases}
\end{equation}
We have put $\nu=1$ for simplicity and consider $b, c\ge 0$.

We give the following definition of solution: this is a weak (or
distributional) solution 
from the point of view of PDE's and a strong solution from the point
of view of stochastic equations.
\begin{defin}
Given $\left(\Omega,\mathcal{F},(\mathcal{F}_{t})_{t\ge 0},
\po\right)$ and an $L_2$-cylindrical Wiener process $\vec w$,
we say that a process $\vec \xi$ is a basic solution to system
\eqref{vort}  on the finite time interval $[0,T]$ with initial
condition $\vec \xi(0)=\vec \xi_0\in L_2$ if 
\begin{equation}\label{equ-def}
\vec \xi \in C([0,T];L_2)\cap L^1(0,T;L_3) \qquad \po-a.s.
\end{equation}
and it satisfies the first equation of \eqref{vort} in the following
sense:\\
for any $t\in [0,T]$, for any $\vec \phi \in H^{2+2c}\cap H^{4-2b}$
\begin{multline}
\langle \vec \xi(t),\vec \phi\rangle
+\int_0^t \langle\vec \xi(s),(-\Delta)^{1+ c }\vec \phi \rangle ds
-\int_0^t \langle(\vec v(s)\cdot \nabla)\vec \phi,\vec  \xi(s)\rangle ds
\\+\int_0^t \langle (\vec \xi(s) \cdot \nabla) \vec \phi, \vec v(s)\rangle ds
=\langle \vec \xi_0,\vec \phi\rangle+
\langle (-\Delta)^{-2}\vec w(t), (-\Delta)^{2-b}\vec \phi\rangle
\end{multline}
$\po$-a.s.
\end{defin}

The latter relationship is obtained by multiplying 
the first equation of \eqref{vort}  by $\vec \phi$, integrating in space and time and finally by 
integration by part in the trilinear terms. Indeed, 
$-\langle(\vec v(s)\cdot \nabla)\vec  \phi,\vec \xi(s)\rangle=
\langle(\vec v(s)\cdot \nabla)\vec \xi(s),\vec  \phi\rangle
=\langle P[(\vec v(s)\cdot \nabla)\vec \xi(s)],\vec
\phi\rangle$ and
$\langle (\vec \xi(s) \cdot \nabla) \vec \phi, \vec v(s)\rangle
=- \langle (\vec \xi(s) \cdot \nabla) \vec v(s), \vec \phi\rangle
= -\langle P[(\vec \xi(s) \cdot \nabla) \vec v(s)], \vec \phi\rangle$,
since $\vec \phi$ is a divergence free vector.

\begin{remark}
We remark that all the terms in \eqref{equ-def} are
meaningful. 
We show the basic estimates for the trilinear terms, by means of H\"older and Sobolev inequalities:
\[\begin{split}
\left|\int_0^t \langle(\vec v(s)\cdot \nabla)\vec \phi,\vec \xi(s)\rangle
ds \right|
&\le 
\|\vec \phi\|_{H^1}\int_0^t \|\vec v(s)\|_{L_6}\|\vec \xi(s)\|_{L_3}ds
\\&\le C
\|\vec \phi\|_{H^1}\int_0^t \|\vec v(s)\|_{H^1}\|\vec \xi(s)\|_{L_3}ds
\\&\le C
\|\vec \phi\|_{H^1}\int_0^t \|\vec \xi(s)\|_{L_2}\|\vec \xi(s)\|_{L_3}ds
\\&\le C
\|\vec \phi\|_{H^1} \|\vec \xi\|_{L^\infty(0,T;L_2)} 
 \|\vec \xi\|_{L^1(0,T;L_3)}
\end{split}
\]
and similarly
\[\begin{split}
\left|\int_0^t \langle (\vec \xi(s) \cdot \nabla) \vec \phi, 
\vec v(s)\rangle ds\right|
&\le C
\|\vec \phi\|_{H^1}\int_0^t \|\vec \xi(s)\|_{L_3}\|\vec v(s)\|_{L_6}ds
\\&\le C
\|\vec \phi\|_{H^1} 
 \|\vec \xi\|_{L^1(0,T;L_3)} \|\vec \xi\|_{L^\infty(0,T;L_2)} 
\end{split}
\]
\end{remark}

Here and in the following, 
we denote by $C$ a generic constant, which may vary from
line to line. However a subscript denotes that the constant depends on
the specified parameters.

\begin{remark}
To prove the well posedness of system \eqref{vort},
we shall exploit the pathwise technique used the first time in \cite{BT} 
and later on in a more useful way in \cite{fla-diss}. We shall transform the 
stochastic equation of It\^o type \eqref{vort} into a random equation 
which behaves like a deterministic equation when studied for $\po$-a.e. 
$\omega \in \Omega$, that is we find estimates for the paths of the 
solution process.

The solution process will enjoy more properties as a stochastic process; 
as in the 2D setting, we shall prove pathwise uniqueness and continuous 
dependence on the initial data in $L_2$. Thus our solution will be a strong 
solution from the point of view of stochastic differential equations (see 
e.g. \cite{IW}), and a Feller and Markov process in $L_2$. 
For these details, see \cite{cime} and references therein.
\end{remark}

\section{Estimates of the nonlinearities}\label{sec-stime}
This is a technical section, where
 we present the estimates to be used in proving the well posedness of
 system \eqref{vort} and \eqref{sis-eta}.

First, we present a classical result.
\begin{lemma}
Let $\vec u, \vec v, \vec w:\mathbb R^3 \to \mathbb
R^3$ be smooth $D$-periodic and divergence free vector fields.
Then 
\begin{equation}\label{trilin}
\langle P[(\vec u\cdot \nabla)\vec v],\vec w\rangle=
-\langle P[(\vec u\cdot \nabla)\vec w],\vec v\rangle
\end{equation}
In particular
\begin{equation}\label{trilin-0}
\langle P[(\vec u\cdot \nabla)\vec v],\vec v\rangle=0
\end{equation}
\end{lemma}
\begin{proof}
First
\[
\langle P[(\vec u\cdot \nabla)\vec v],\vec w\rangle=
\langle (\vec u\cdot \nabla)\vec v,\vec w\rangle=
\sum_{i,j=1}^3\int_{D}
u^{(i)}(\vec x) \partial_i v^{(j)}(\vec x) w^{(j)}(\vec x) \ d\vec x
\]
Then by integration by parts we get \eqref{trilin}. The relationship 
\eqref{trilin-0} is obtained from \eqref{trilin} by taking $\vec
w=\vec v$.
\end{proof}
By density, the above results hold for all vectors giving meaning to
the above expressions. One can find estimates on the trilinear term in \cite{temam2}. Here we present particular estimates, not included in \cite{temam2}, and useful in the sequel.
Their proofs are based on Sobolev embeddings theorems
and  H\"older inequalities. 
\begin{lemma}\label{lemmm}
Let $ c \ge 0$.
Then there exists a positive constant $C$ (depending on $c$) such that for any $\epsilon>0$ we have
\begin{equation}\label{st1}
|\langle (\vec u_1\cdot  \nabla) \vec u_2, \vec u_3 \rangle|
\le
\epsilon \|\vec u_2\|_{H^{1+c}}^2
+ \frac {C}{ \epsilon}
\|\vec u_1\|_{H^{1}}^2 \|\vec u_3\|^2_{L_3} 
\end{equation}
\begin{equation}\label{st3}
|\langle (\vec u_1\cdot  \nabla) \vec u_2, \vec u_3 \rangle|
\le
\epsilon \|\vec u_3\|_{H^{1+c}}^2
+ \frac {C}{ \epsilon}
 \|\vec u_1\|^2_{L_3} \|\vec u_2\|_{H^{1+c}}^2
\end{equation}
\begin{equation}\label{st2}
|\langle (\vec u_1 \cdot \nabla )\vec u_2, \vec u_3\rangle|
\le
\epsilon \big\|\vec u_3 \big\|_{H^{1+ c }}^2
+ \frac {C}{\epsilon}
  \|\vec u_1\|_{H^1}^2 \|\vec u_2\|^2_{L_3}
\end{equation}
for all vectors making finite each  r.h.s.
\end{lemma}
\begin{proof} We begin with the first inequality:
\[\begin{split}
|\langle (\vec u_1 \cdot \nabla )\vec u_2, \vec u_3\rangle|
&\le \|\vec  u_1\|_{L_6} \|\nabla \vec u_2\|_{L_2} \|\vec u_3\|_{L_3}
\text{ by H\"older inequality}
\\
&\le C  \|\vec u_1\|_{H^1} \|\vec u_2 \|_{H^1} \|\vec u_3\|_{L_3}
\text{ by Sobolev embedding } H^1\subset L_6
\\&\le C_c \|\vec u_1 \|_{H^1}\|\vec u_2\|_{H^{1+c}}\|\vec u_3\|_{L_3}
\\&\le \epsilon \|\vec u_2\|_{H^{1+c}}^2+\frac {C_c^2}{4\epsilon}\|\vec u_1 \|^2_{H^1}\|\vec u_3\|_{L_3}^2
\text{ by Cauchy inequality}
\end{split}\]

For the second inequality, we proceed in a similar way:
\[
\begin{split}
|\langle (\vec u_1\cdot  \nabla) \vec u_2, \vec u_3 \rangle|
&\le
\|\vec  u_1\|_{L_3} 
\|\nabla \vec u_2\|_{L_2} 
\|\vec u_3\|_{L_6}
\\&\le
C \|\vec u_1\|_{L_3} \|\vec u_2 \|_{H^1} \|\vec u_3\|_{H^1}
\\&\le
C_ c  \|\vec u_1\|_{L_3} \|\vec u_2\|_{H^{1+c}} \|\vec u_3\|_{H^{1+ c }}
\end{split}\]
Then we apply Cauchy inequality to get the desired result.

For the third inequality, we have
\[
\langle (\vec u_1 \cdot \nabla )\vec u_2, \vec u_3\rangle
=
-\langle (\vec u_1 \cdot \nabla )\vec u_3, \vec u_2\rangle
\]
from \eqref{trilin}.
Then we get \eqref{st2} from \eqref{st1}.
\end{proof}

\begin{lemma}\label{st-difficile}
Let $ c  \ge \frac 14$.
Then there exists  a positive constant $C$ 
(depending on $c$) such that for any $\epsilon>0$ we have
\[
 |\langle (\vec u_1 \cdot \nabla) \vec u_2, \vec u_1 \rangle|
 \le 
 \epsilon \big\|\vec u_1\big\|^2_{H^{1+ c }  }
 +\frac C{\epsilon} 
  \big\|\vec u_2 \big\|_{H^{1+c}} ^2 \big\|\vec u_1\big\|_{L_2}^2
\]
for all vectors making finite the r.h.s..
\end{lemma}
\begin{proof} First we consider the range of values 
$\frac 14 \le  c <\frac 12$. We have
$\frac {1-2 c }6 +\frac {3-2 c }6 +\frac 12 \le 1$ and 
$H^{1+ c } \subset  L_{\frac 6{1-2 c }}$,
$ H^ c  \subset L_{\frac 6{3-2 c }}$. Thus, H\"older
and Sobolev inequalities give
\[
\begin{split}
|\langle  (\vec u_1 \cdot \nabla) \vec u_2, \vec u_1 \rangle|
&\le
\|\vec u_1\|_{L_{\frac 6{1-2 c }}}
\| \nabla \vec u_2\|_{L_{\frac 6{3-2 c }}}
\|\vec u_1\|_{L_2}
\\&\le
C_{ c } \|\vec u_1\|_{H^{1+ c }}
 \| \nabla \vec u_2\|_{H^ c }
\|\vec u_1\|_{L_2}
\\&\le 
C_{ c } \|\vec u_1\|_{H^{1+ c }}
\|\vec u_2\|_{H^{1+c}}
\|\vec u_1\|_{L_2}
\end{split}\]

Otherwise, for $c \ge \frac 12$,  we use the Sobolev embeddings  $H^{\frac 12}\subset L_3$ and
$H^1\subset L_6$. Therefore, again we estimate
\[
\begin{split}
|\langle (\vec u_1 \cdot \nabla) \vec u_2, \vec u_1\rangle|
&\le
\|\vec u_1\|_{L_6}
\|\nabla \vec u_2\|_{L_3}
\|\vec u_1\|_{L_2}
\\&\le C
\|\vec u_1\|_{H^1}
\|\nabla \vec u_2\|_{H^{\frac 12}}
\|\vec u_1\|_{L_2}
\\&\le
C_{c } 
\|\vec u_1\|_{H^{1+c}}\|\vec u_2\|_{H^{1+c}}
\|\vec u_1\|_{L_2}
\end{split}\]

Applying Cauchy inequality  we conclude the proof. \end{proof}

\section{The linear equation}\label{sec-ou}
When we neglect the non linearites in system \eqref{vort} for the
vorticity, we get
\begin{equation}\label{ou}
\begin{cases}
d\vec \zeta +(-\Delta)^{1+ c }\vec \zeta \ dt=(-\Delta)^{-b} d\vec w
\\
\nabla \cdot \vec \zeta=0
\end{cases}
\end{equation}
Here the second equation keeps track of the fact that the
vorticity vector is divergence free. So $\vec \zeta$  is the usual
Ornstein-Uhlenbeck process, well studied in the literature.
Here we assume 
$\vec \zeta(0)=\vec 0$. Therefore the mild solution of \eqref{ou} is 
\begin{equation}\label{zeta-mild}
 \vec \zeta(t)=\int_0^t e^{-(-\Delta)^{1+ c }(t-s)} (-\Delta)^{-b}d\vec w(s) 
\end{equation}
(see e.g. \cite{dpz}). We have 
\begin{prp}\label{pro-zeta}
Let 
\begin{equation}\label{ipo}
2b+ c >a+\frac 12
\end{equation}
Then,  for any $m\in \mathbb N$
\[ \vec \zeta\in C(\mathbb R_+;H_{2m}^a) \qquad \po-a.s
\] 
\end{prp}
\begin{proof}
The proof is basically the same as that in \cite{dp}
proving that $\vec \zeta $ has $\po$-a.e. path 
in $C(\mathbb R_+;H^a)$. Working on the torus, we can
improve that result  getting $\vec \zeta \in C(\mathbb R_+;H^a_{2m})$. 

The factorization method uses that 
\begin{equation}\label{eq-z-y}
\vec \zeta(t)=\frac{\sin(\pi\alpha)}{\pi}
\int_0^t \frac1{(t-s)^{1-\alpha}}e^{-(-\Delta)^{1+ c }(t-s)} \vec Y_\alpha(s)ds
\end{equation}
for $0<\alpha<1$, with 
\[
\vec Y_\alpha(s)=\int_0^s \frac{1}{(s-r)^\alpha} e^{-(-\Delta)^{1+ c }(s-r)} 
(-\Delta)^{-b}d\vec w(r)
\]
Now we prove that
under assumption \eqref{ipo} there exists $\alpha\in (0,\frac 12)$
such that
\begin{equation}\label{sultoro}
\mathbb E\|\vec Y_\alpha\|^{2m}_{L^{2m}(0,T;H^a_{2m})}<\infty
\end{equation}
for any $m \in \mathbb N$.

For fixed $\vec x$ and $t$, 
$[(-\Delta)^{a/2}\vec Y_\alpha](t,\vec x)$ 
is a Gaussian random variable given by the
sum of independent Gaussian random variables
\[
(-\Delta)^{a/2}\vec Y_\alpha(t,\vec x)=\sum_{\vec k\in \mathbb Z^3_0} 
 |\vec k|^{a}  \sum_{j=1}^2
\int_0^t \frac 1{(t-s)^\alpha}e^{-|\vec k|^{2(1+ c )}(t-s)} |\vec k|^{-2b} 
      \vec b_{\vec k,j}d\beta_{\vec k,j}(s) e^{i \vec k \cdot \vec x}
\]
Therefore the variance of 
$(-\Delta)^{a/2}\vec Y_\alpha(t,\vec x)$ is the sum of the
variance of each addend:
\[\begin{split}
\mathbb E |(-\Delta)^{a/2}\vec Y_\alpha(t,\vec x)|^2
&=\sum_{\vec k\in \mathbb Z^3_0} |\vec k|^{2a-4b}
\int_0^t \frac 1{(t-s)^{2\alpha}}
e^{-2|\vec k|^{2(1+ c )}(t-s)} ds
\\&=
\sum_{\vec k\in \mathbb Z^3_0} |\vec k|^{2a-4b}
\int_0^t \frac 1{r^{2\alpha}}
e^{-2|\vec k|^{2(1+ c )}r} dr
\\&=\sum_{\vec k\in \mathbb Z^3_0}  |\vec k|^{2a-4b}
|\vec k|^{2(1+ c )(2\alpha-1)} \int_0^{t|\vec k|^{2(1+ c )}}\frac 1{u^{2\alpha}}e^{-2u}du
\\&\le
\sum_{\vec k\in \mathbb Z^3_0}  |\vec k|^{2a-4b}
|\vec k|^{2(1+ c )(2\alpha-1)} \int_0^\infty\frac 1{u^{2\alpha}}e^{-2u}du
\\&= C_\alpha
\sum_{\vec k\in \mathbb Z^3_0} |\vec k|^{2a-4b+2(1+ c )(2\alpha-1)} 
\end{split}
\]
where the constant $C_\alpha$ is finite for any $\alpha<\frac 12$.

Since $(-\Delta)^{a/2}\vec Y_\alpha(t,\vec x)$ 
is a centered Gaussian  random variable, for any integer $m$ we have
\[
\mathbb E |(-\Delta)^{a/2}\vec Y_\alpha(t,\vec x)|^{2m}=
C_m \left(\mathbb E |(-\Delta)^{a/2}\vec Y_\alpha(t,\vec x)|^2\right)^m
\le
C_{m,\alpha}
\left(\sum_{\vec k\in \mathbb Z^3_0} |\vec k|^{2a-4b+2(1+ c )(2\alpha-1)} \right)^m
\]
Integrating with respect to the variables $t\in [0,T]$ and $\vec x\in D$ we get
\[
\mathbb E\|\vec Y_\alpha\|_{L^{2m}(0,T;H^a_{2m})}^{2m}\le C_{m,\alpha}
T (2\pi)^3
\left(\sum_{\vec k\in \mathbb Z^3_0} |\vec k|^{2a-4b+2(1+ c )(2\alpha-1)} \right)^m
\]
The series in the r.h.s. converges if and only if 
\[
2a-4b+2(1+ c )(2\alpha-1)<-3
\]
i.e.
\begin{equation}\label{ipo2}
2b+ c >a+\frac 12+2\alpha(1+ c )
\end{equation}

If \eqref{ipo} holds then there exists $\alpha>0$ small
enough to get \eqref{ipo2} 
and thus for such an $\alpha$ we have proved \eqref{sultoro}.

Now, given \eqref{sultoro},  with a trivial modification  of
the proof of Lemma 2.7 in \cite{dp}, from \eqref{eq-z-y} we get 
\[
\mathbb E \sup_{0\le t\le T} \|\vec \zeta(t)\|_{H^a_{2m}}^{2m}\le 
C_{m,T} \mathbb \|\vec Y_\alpha\|^{2m}_{L^{2m}(0,T;L_{2m})}
\]
and the continuity result.
\end{proof}

\section{The vorticity transport equation}\label{sec-eta}
As explained before, we consider the system obtained from \eqref{vort} by neglecting the term 
$P[(\vec \xi \cdot \nabla) \vec v]$ in the first equation.
This is
\begin{equation}\label{burg}
\begin{cases}
 d\vec \eta+ (-\Delta)^{1+ c }\vec \eta \ dt 
+ P[(\vec v\cdot \nabla)\vec \eta] \ dt 
 =(-\Delta)^{-b}d\vec w
\\
\nabla \cdot \vec v =0
\\
\vec \eta=\nabla \times \vec v
\end{cases}
\end{equation}
We call it the vorticity transport system, since its first equation
 is a reduced form 
of the vorticity equation in \eqref{vort}: in \eqref{burg} vorticity
is only transported, not stretched.

Let us point out a feature of the equation of $\vec \eta$.
The nonlinearity $(\vec v\cdot \nabla)\vec \eta$ has a peculiar form
similar to that appearing 
in the regularized form of Leray-$\alpha$ models for fluids (see e.g. \cite{BBF}), that is the first
entry of the bilinear term $P[(\vec v\cdot\nabla)\vec \eta]$ 
is not the unknown
$\vec \eta$ itself but indeed $\vec v$, which has one order more of regularity
with respect to $\vec \eta$ (recall that if $\vec \eta \in H^b_p$ then
$\vec v \in H^{b+1}_p$). Therefore,
even if $\vec \eta$ satisfies a nonlinear equation, 
the quadratic term $(\vec v\cdot \nabla)\vec
\eta$ in \eqref{burg} 
(with $\vec \eta=\nabla \times \vec v$)
behaves better than 
$(\vec v\cdot \nabla)\vec v$ in \eqref{vel}
and this makes the difference in the
analysis of systems \eqref{burg} and \eqref{vel}.

As far as the technique is concerned, we point out that
in order to get existence and uniqueness results, we could look for
mean estimates.
However, for our purpose it is enough to get pathwise estimates (see
Theorem \ref{gener}).
Moreover, the advantage of the pathwise approach is twofold: 
the existence result is obtained asking weaker assumption on the
covariance of the noise and the regularity results are easily
obtained.
To see the first advantage,  thanks to \eqref{trilin}, 
with the usual techniques  (see e.g. \cite{BT}, \cite{cime})
we can get
\[
\mathbb E\left[ \| \vec \eta (t)\|_{L_2} ^{2}+2
\int_{0}^{t}\| \vec \eta (s)\|_{H^{1+ c }}^{2}ds\right]
\le 
\|\vec \eta (0)\|_{L_2}^{2}+Tr\left( (-\Delta)^{-2b}\right)
\ t
\]
This requires $Tr\left( (-\Delta)^{-2b}\right)<\infty$, i.e.
\[
\sum_{\vec k \in \mathbb Z^3_0}|\vec k|^{-4b}<\infty
\]
which holds when $b>\frac 34$.
But Theorem \ref{eta-completo} allows to get existence of a basic
solution $\vec \eta$ for $b>\frac 14-\frac c2$.
Since our task in Theorem \ref{gener} will be to estimate
\[
\|(-\Delta)^b P[(\vec \eta\cdot \nabla )\vec v]\|_{L_2}
\]
it is clear than the smaller is $b$
the easier is our task.

For this aim, we set
$\vec \beta=\vec \eta-\vec \zeta$ and exploit that the noise is
independent of the unknowns; then
\begin{equation}\label{eq-beta}
\begin{cases}
\dfrac{\partial\vec \beta}{\partial t}  + (-\Delta)^{1+ c }\vec\beta
+P[(\vec  v\cdot \nabla)(\vec  \beta+\vec \zeta)]=\vec 0
\\
\nabla \cdot \vec v=0
\\
\nabla \times \vec v=\vec \beta +\vec \zeta
\end{cases}
\end{equation}

System \eqref{eq-beta} is studied pathwise. We have the following result
\begin{prp}\label{pro-beta}
i) Assume that
\[
\begin{cases}
 c  \ge 0\\
2b+ c >\frac 12
\end{cases}
\]
Then, for any $\vec \beta(0)\in L_2$
there exists a solution to \eqref{eq-beta} such
that 
$$\vec \beta \in  
C([0,T];L_2) \cap  L^2(0,T;H^{1+ c })
\qquad\po-a.s.$$
ii)
Assume that
\[
\begin{cases}
 c  \ge 0\\
2b+ c >\frac 32
\end{cases}
\]
Then, for any  $\vec \beta(0)\in H^1$
the solution  given in i)  enjoys also
$$\vec \beta \in  
C([0,T];H^1) \cap  L^2(0,T;H^{2+ c })
\qquad\po-a.s.$$
iii) 
Assume that
\[
\begin{cases}
 c  \ge 0\\
2b+ c >\frac 52
\end{cases}
\]
Then, for any $\vec \beta(0)\in H^2$
the solution given in i) enjoys also
$$\vec \beta \in  
C([0,T];H^2) \cap  L^2(0,T;H^{3+ c })
\qquad\po-a.s.$$
\end{prp}
\begin{proof} 
We proceed pathwise.
The technique to prove existence is to consider first the 
finite dimensional problem, obtained by applying the projection operator
$\Pi_n$ to \eqref{eq-beta}. The goal is to find suitable a priori
estimates, uniformly in $n$.
Thus, when any finite dimensional (Galerkin) problem has a solution
we pass to the limit as $n \to \infty$ to get an existence result 
for \eqref{eq-beta}.
This technique, based on finite dimensional approximation, is well
known (see e.g. \cite{temam, temam2}).
Therefore we look for a priori estimates for the full system
\eqref{eq-beta}; they hold for any Galerkin approximation as well, but
we  skip the details for the limit as $n \to \infty$.

i) We multiply the l.h.s. of the first equation of \eqref{eq-beta}
by $\vec\beta(t)$ and integrate over $D$.
Using \eqref{trilin}-\eqref{trilin-0} and then H\"older and Sobolev inequalities, we get
\[\begin{split}
\frac 12 \frac{d}{dt}\|\vec \beta(t)\|_{L_2}^2
+ \|\vec \beta(t)\|_{H^{1+c}}^2
&=
-\langle P[(\vec v(t)\cdot \nabla)\vec \zeta(t)],\vec \beta(t)\rangle
\\&
=\langle (\vec v(t)\cdot \nabla)\vec \beta(t),\vec \zeta(t)\rangle
\\&
\le C \|\vec v(t)\|_{L_6} \|\vec \beta(t)\|_{H^1}\|\vec \zeta(t)\|_{L_3}
\\&
\le C_c \|\vec v(t)\|_{H^1} \|\vec \beta(t)\|_{H^{1+c}}\|\vec \zeta(t)\|_{L_3}
\\&
\le C \|\vec \beta(t)+\vec \zeta(t)\|_{L_2} \|\vec\beta(t)\|_{H^{1+c}}
    \|\vec \zeta(t)\|_{L_3}
\end{split}
\]
Cauchy inequality gives
\begin{equation}\label{1stimab}
\frac 12 \frac{d}{dt}\|\vec \beta(t)\|_{L_2}^2
+ \|\vec \beta(t)\|_{H^{1+c}}^2
\le
\frac 12 \|\vec\beta(t)\|_{H^{1+c}}^2
+C \|\vec \zeta(t)\|_{L_3}^2   \|\vec \beta(t)\|_{L_2}^2  
+C \|\vec \zeta(t)\|_{L_3}^4
\end{equation}
Therefore, Gronwall inequality applied to 
\[
\frac{d}{dt}\|\vec \beta(t)\|_{L_2}^2
\le
C \|\vec \zeta(t)\|_{L_3}^2   \|\vec \beta(t)\|_{L_2}^2  
+C \|\vec \zeta(t)\|_{L_3}^4
\]
gives
\[
\sup_{0\le t\le T}\|\vec \beta(t)\|_{L_2}^2
\le C(b,c ,T,\|\vec \beta(0)\|_{L_2},\|\vec \zeta\|_{L^\infty(0,T;L_3 )})
\]
Integrating in time \eqref{1stimab} we get 
\[
  \int_0^T\|\vec \beta(t)\|_{H^{1+ c }}^2 dt \le 
\tilde C(b, c ,T,\|\vec \beta(0)\|_{L_2},\|\vec \zeta\|_{L^\infty(0,T;L_3)} )
\]

We remind that $\vec \zeta \in C([0,T];L_3)$ if $2b+c>\frac 12$,
according to Proposition \ref{pro-zeta}. 
Then these a priori estimates give 
$\vec \beta \in L^\infty(0,T;L_2) \cap L^2(0,T;H^{1+ c })$.

Moreover,
\[
\frac{\partial \vec \beta}{\partial t}=- (-\Delta)^{1+ c }\vec \beta
-P[(\vec  v\cdot \nabla) \vec \beta]
-P[(\vec  v\cdot \nabla) \vec \zeta]
\]
Given the regularity of $\vec \beta$ we have that the r.h.s. belongs to 
$L^2(0,T;H^{-1- c })$; indeed $(-\Delta)^{1+ c }\vec \beta
\in L^2(0,T;H^{-1- c })$ and the two latter terms belong to 
$L^2(0,T;H^{-1})$. Let us see this; we proceed as before
\[
|\langle(\vec  v\cdot \nabla) \vec \beta, \vec u\rangle|=
|\langle (\vec  v\cdot \nabla) \vec u, \vec \beta\rangle|
\le
\|\vec v\|_{L_6}\|\nabla \vec u\|_{L_2}\|\vec \beta\|_{L_3}
\]
This gives 
\[\begin{split}
\|(\vec  v\cdot \nabla) \vec \beta\|_{H^{-1}}=
\sup_{\|\vec u\|_{H^1}> 0}
\frac {|\langle(\vec  v\cdot \nabla) \vec \beta, 
\vec u\rangle|}{\|\vec u\|_{H^1}}
& \le \|\vec v\|_{L_6}\|\vec \beta\|_{L_3}
\\
&\le C \|\vec v\|_{H^1} \|\vec \beta\|_{H^1}
\\
&\le C (\|\vec \beta\|_{L_2}+\|\vec \zeta\|_{L_2}) \|\vec \beta\|_{H^1}
\end{split}\]
Similarly we deal with $(\vec  v\cdot \nabla) \vec \zeta$:
\[
\|(\vec  v\cdot \nabla) \vec \zeta\|_{H^{-1}}\le
\|\vec v\|_{L_6}\|\vec \zeta\|_{L_3}
\le
C \|\vec \zeta\|_{L_3}^2+\|\vec \zeta\|_{L_3}\|\vec \beta\|_{L_2}
\]

We recall 
that the space $\{\vec\beta \in L^2(0,T;H^{1+ c }): 
\frac{\partial \vec \beta}{\partial t}\in L^2(0,T;H^{-1- c })\}$
is compactly embedded in $L^2(0,T;L_2)$.

These are the basic results to implement the Galerkin approximation.

As far as
the continuity is concerned,
the fact that $\vec \beta \in L^2(0,T;H^{1+ c })$ and 
$\frac{\partial \vec \beta}{\partial t}\in L^2(0,T;H^{-1- c })$ implies 
$\vec \beta \in C([0,T];L_2)$ (see Ch III Lemma 1.2 of \cite{temam}).

ii) We need a priori estimates and we proceed as in the previous step.
We multiply the l.h.s. of the first equation of \eqref{eq-beta}
by $-\Delta \vec\beta(t)$ and integrate on $D$. We get
\[
\frac 12 \frac{d}{dt}\|\vec \beta(t)\|_{H^1}^2
+ \|\vec \beta(t)\|_{H^{2+c}}^2
=\langle (\vec  v(t)\cdot \nabla)(\vec  \beta(t)
+ \vec \zeta(t)),\Delta \vec\beta(t)\rangle
\]
We estimate the r.h.s. as follows
\[\begin{split}
\langle (\vec  v\cdot \nabla)(\vec  \beta+ \vec \zeta),\Delta \vec\beta\rangle
&\le
\|(\vec v\cdot \nabla)(\vec\beta+\vec \zeta)\|_{L_2}\|\Delta \vec\beta\|_{L_2}
\\&
\le \|\vec v\|_{L_\infty}\|\vec\beta+\vec \zeta\|_{H^1}\|\vec\beta\|_{H^2}
\\&
\le C \|\vec v\|_{H^2}\|\vec\beta+\vec\zeta\|_{H^1}\|\vec\beta\|_{H^2}
\;\text{ since } H^2\subset L_\infty
\\&
\le C_c \|\vec \beta+\vec \zeta\|_{H^1}^2\|\vec\beta\|_{H^{2+c}}
\\&
\le \frac 12 \|\vec\beta\|_{H^{2+c}}^2+C \|\vec \beta\|_{H^1}^4+\|\vec \zeta\|_{H^1}^4
\end{split}
\]
This gives
\[
\frac{d}{dt}\|\vec \beta(t)\|_{H^1}^2
+ \|\vec \beta(t)\|_{H^{2+c}}^2\le
C \|\vec \beta\|_{H^1}^4+\|\vec \zeta\|_{H^1}^4
\]
and we conclude as before using Gronwall Lemma and the fact that
$\vec \beta \in L^2(0,T;H^1)$ from i) and $ \vec \zeta \in
C([0,T];H^1)$ from Proposition \ref{pro-zeta}, getting
\[
\sup_{0\le t\le T}\|\vec \beta(t)\|_{H^1}^2
\le C(b,c ,T,\|\vec \beta(0)\|_{H^1},\|\vec \zeta\|_{L^\infty(0,T;H^1)})
\]
\[
\int_0^T\|\vec \beta(t)\|_{H^{2+ c }}^2 dt \le 
\tilde C(b, c ,T,\|\vec \beta(0)\|_{H^1},\|\vec \zeta\|_{L^\infty(0,T;H^1)} )
\]
Continuity in time is obtained as before.

iii) We multiply the l.h.s. of the first equation of \eqref{eq-beta}
by $(-\Delta)^2 \vec\beta(t)$ and integrate on $D$. We get
\[
\frac 12 \frac{d}{dt}\|\vec \beta(t)\|_{H^2}^2
+ \|\vec \beta(t)\|_{H^{3+c}}^2
=-\langle (\vec  v(t)\cdot \nabla)(\vec  \beta(t)
+ \vec \zeta(t)),(-\Delta)^2 \vec\beta(t)\rangle
\]
We estimate the r.h.s. as follows. First, we use the estimate for the product; by means of 
the Sobolev embedding $H^2\subset L_\infty$ we get
\[\begin{split}
\|fg\|_{H^1}
\le \|g\nabla f\|_{L_2}+\|f\nabla g\|_{L_2}
&\le \|\nabla f\|_{L_\infty} \|g\|_{L_2}+\|f\|_{L_\infty}\|\nabla g\|_{L_2}
\\&\le C \|f\|_{H^3} \|g\|_{L_2}+C\|f\|_{H^2}\|g\|_{H^1}
\end{split}\]
Hence, for the trilinear term we get
\[\begin{split}
\langle (\vec v\cdot\nabla)(\vec\beta+ \vec\zeta),(-\Delta)^2 \vec\beta \rangle
&=\langle (-\Delta)^{\frac 12}[(\vec v\cdot\nabla)(\vec\beta+\vec\zeta)],
(-\Delta)^{\frac 32} \vec\beta \rangle
\\&
\le \|(\vec v\cdot\nabla)(\vec\beta+\vec\zeta)\|_{H^1}\|\vec\beta\|_{H^3}
\\&
\le C \Big(\|\vec v\|_{H^3} \|\vec\beta+\vec\zeta\|_{H^1}
+\|\vec v\|_{H^2}\|\vec\beta+\vec\zeta\|_{H^2}\Big)
\|\vec\beta\|_{H^3}
\\&
\le C \|\vec\beta+\vec\zeta\|_{H^1} \|\vec\beta+\vec\zeta\|_{H^2}
\|\vec\beta\|_{H^3}
\\&
\le C_c \|\vec\beta+\vec\zeta\|_{H^1} \|\vec\beta+\vec\zeta\|_{H^2}
\|\vec\beta\|_{H^{3+c}}
\\
&\le \frac 12 \|\vec\beta\|_{H^{3+c}}^2
+C  \|\vec\beta+\vec\zeta\|^2_{H^1} \|\vec \beta\|_{H^2}^2
+C \|\vec\beta+\vec\zeta\|^2_{H^1} \|\vec \zeta\|_{H^2}^2
\end{split}\]
This gives
\[
\frac{d}{dt}\|\vec \beta(t)\|_{H^2}^2
+ \|\vec \beta(t)\|_{H^{3+c}}^2
\le
C \|\vec\beta(t)+\vec\zeta(t)\|^2_{H^1} \|\vec \beta(t)\|_{H^2}^2
+C \|\vec \zeta(t)\|_{H^2}^2 \|\vec \beta(t)\|_{H^1}^2+C\|\vec \zeta(t)\|_{H^2}^4
\]
Since  $\vec \beta\in C([0,T];H^1)$ from step ii) 
and $\vec \zeta \in C([0,T];H^2)$ from Proposition \ref{pro-zeta}, 
we get first
\[
\sup_{0\le t\le T}\|\vec \beta(t)\|_{H^2}^2
\le C(b,c ,T,\|\vec \beta(0)\|_{H^2},\|\vec \zeta\|_{L^\infty(0,T;H^2)})
\]
and then
\[
\int_0^T\|\vec \beta(t)\|_{H^{3+ c }}^2 dt \le 
\tilde C(b, c ,T,\|\vec \beta(0)\|_{H^2},\|\vec \zeta\|_{L^\infty(0,T;H^2)} )
\]
Continuity in time is obtained as before.
This concludes the proof.
\end{proof}

Now we come back to the unknown $\vec \eta=\vec \beta+\vec \zeta$.
The definition of basic solution is the same as that for $\vec \xi $ given at
the end of Section \ref{sec:ma}, with the obvious modification of the
equation by  neglecting $P[(\vec \xi\cdot \nabla)\vec v]$.
\begin{thm}\label{eta-completo}
i) Assume that
\[
\begin{cases}
 c  \ge 0\\
2b+ c >\frac 12
\end{cases}
\]
Then, for any $\vec\eta(0)\in L_2$ there exists a unique process 
$\vec \eta$ which is a basic solution to \eqref{burg} such that
\[
 \vec \eta \in 
C([0,T];L_2)  \cap L^2(0,T;L_6)
\]
$\po$-a.s.
\\
Moreover there is continuous dependence on the initial data: given two
initial data $\vec \eta(0),\vec \eta_\star(0)\in L_2$ we have
\[
 \|\vec \eta(0)-\vec \eta_\star(0)\|_{L_2} \to 0 \Longrightarrow 
 \|\vec\eta-\vec \eta_\star\|_{C([0,T];L_2)} \to 0 
\]
ii)
Assume that
\[
\begin{cases}
 c  \ge 0\\
2b+ c >\frac 32
\end{cases}
\]
Then, for any  $\vec \eta(0)\in H^1$
the solution  given in i)  enjoys also
$$
\vec \eta \in   C([0,T];H^1)\qquad\po-a.s.
$$
iii) 
Assume that
\[
\begin{cases}
 c  \ge 0\\
2b+ c >\frac 52
\end{cases}
\]
Then, for any $\vec \eta(0)\in H^2$
the solution given in i) enjoys also
$$
\vec \eta \in  C([0,T];H^2) \qquad\po-a.s.
$$
\end{thm}
\begin{proof}
The existence comes from the existence results on  $\vec \beta$,
$\vec \zeta$. Moreover
\[
 \vec \zeta \in C([0,T];L_q) \qquad \forall q<\infty
\]
and by Sobolev embedding
\[
\vec \beta \in L^2(0,T;H^{1+ c })\subset L^2(0,T;H^{1})\subset L^2(0,T;L_6)
\]
Merging toghether the regularity of these processes we
get our results  for $\vec \eta$.

As far as continuous dependence on the initial data  is concerned, 
let us take two basic solutions $\vec \eta_1$ and $\vec \eta_2$ with 
$\vec \eta_1(0)=\vec \eta_2(0)\in L_2$; at least we have
\[
\vec \eta_1,\vec \eta_2 \in C([0,T];L_2)\cap L^2(0,T;L_3)
\]
We define $\vec y = \vec \eta_1-\vec \eta_2$; then
the system fulfilled by $\vec y$ can be written as
\[
\begin{cases}
\dfrac{\partial\vec y}{\partial t}+( -\Delta)^{1+ c }\vec y 
+P[(\vec v_1\cdot \nabla ) \vec y ]+P[\left((\vec v_1-\vec v_2)\cdot \nabla\right)
\vec \eta_2]=\vec 0
\\
\nabla \cdot \vec v_1=\nabla \cdot \vec v_2=0
\\
\vec y=\nabla \times (\vec v_1-\vec v_2)
\end{cases}
\]
We estimate the following term, as usual:
\[
\begin{split}
 |\langle [(\vec v_1-\vec v_2)\cdot \nabla] \vec \eta_2, \vec y\rangle |
&= |\langle [(\vec v_1-\vec v_2)\cdot \nabla] \vec y,\vec \eta_2\rangle |
\\
& \le \frac 1 2 \|\vec y\|_{H^{1+ c }} ^2 
+ C \|\vec \eta_2\|^2_{L_3} \|\vec v_1-\vec v_2\|^2_{H^1} \text{ from }\eqref{st1}
\\
& \le \frac 1 2 \|\vec y\|_{H^{1+ c }} ^2 
+ C  \|\vec \eta_2\|^2_{L_3} \|\vec y\|^2_{L_2} 
\end{split}
\]
Then taking the scalar product of the  the first equation for $\vec y$
with $\vec y$,  integrating on the spatial domain
and using \eqref{trilin}, we get
\[
\frac{d}{dt} \|\vec y(t)\|^2_{L_2} +  \|\vec y(t)\|^2_{H^{1+c}}
\le  
C  \|\vec \eta_2(t)\|^2_{L_3} \|\vec y(t)\|^2_{L_2} 
\]
Recall that   $\vec \eta_2 \in L^2(0,T;L_3)$.
Applying Gronwall lemma
to 
\[
\frac{d}{dt} \|\vec y(t)\|^2_{L_2} 
\le  
C  \|\vec \eta_2(t)\|^2_{L_3} \|\vec y(t)\|^2_{L_2} 
\]
 we get
$$
\sup_{0\le t\le T}\|\vec y(t)\|_{L_2}\le \|\vec y(0)\|_{L_2}e^{C
  \int_0^T \|\vec \eta_2(t)\|^2_{L_3} dt}
$$
This gives the continuous dependence on the
initial data;   uniqueness is obtained when $\vec y(0)=\vec 0$.
\end{proof}

\section{The vorticity equation}\label{sec-vor}
Now we consider the full nonlinear system \eqref{vort}.
If the initial velocity is more regular, say $\vec
v(0)\in H^1$ (i.e. $\vec \xi(0)\in L_2$),
one can prove a local existence and uniqueness result for $ c =0$;
global existence holds only for $ c \ge \frac 14$ (see \cite{fe-proc}).
In this paper we improve the results for $ c \ge \frac 14$
considering initial data $\vec \xi(0)\in H^1$ and $H^2$.

We need a preliminary result for the velocity, fulfilling \eqref{sis-nonlin}
 with  
the noise obtained from a  Wiener process $\vec w_{vel}$ such that
$\nabla \times \vec w_{vel}=(-\Delta)^{-b}\vec w$, that is 
\[
\vec w_{vel}(t,\vec x)=\sum_{\vec k\in \mathbb Z^3_0} |\vec k|^{-2b-1}
 [-\vec b_{\vec k,1}\beta_{\vec k,1}(t) +\vec b_{\vec
     k,2}\beta_{\vec k,2}(t)]e^{i\vec k \cdot \vec x}
\]
Therefore \eqref{sis-nonlin} becomes
\begin{equation}\label{NSwvel}
\begin{cases}
 d\vec v+ (-\Delta)^{1+ c }\vec v \ dt 
+ (\vec v\cdot \nabla)\vec v \ dt +\nabla p \ dt
 =d\vec w_{vel}
\\
\nabla \cdot \vec v =0
\end{cases}
\end{equation}

\begin{prp}\label{theta-v}
Assume that
\[
\begin{cases}
 c  \ge 0\\
b>\frac 14
\end{cases}
\]
Then for any $\vec v(0)\in L_2$ there exists a process 
$\vec v$ with $\po$-a.e. path in $L^\infty(0,T;L_2)\cap L^2(0,T;H^{1+ c })$,
solving \eqref{NSwvel}.
\end{prp}
\begin{proof}
We know the result for $ c =0$ (see \cite{cime}); the case
$ c >0$ does not provide any difficulty. But we show the
shortest way to get it, by means of mean value estimates. Only here we
use mean value estimates instead of the pathwise ones.

We write the basic energy estimate obtained
from It\^o formula for $d\|\vec v(t)\|_{L_2}^2$; the details can be
found in \cite{cime}. We have
\[
\mathbb E \|\vec v(t)\|_{L_2}^2 +2\int_0^t \mathbb E \|\vec
v(s)\|_{H^{1+ c }}^2ds\le \|\vec v(0)\|_{L_2}^2
+t \sum_{\vec k\in \mathbb Z^3_0} |\vec k|^{-2(2b+1)}
\]
The series in the r.h.s. converges if and only if $2(2b+1)>3$, i.e.  
$b>\frac 14$. These estimates improves the regularity:
$\vec v \in  L^2(0,T;H^{1+ c })$, $\po$-a.s.
\end{proof}

Now we consider the unknown $\vec \xi$. 
Let $\vec \delta:=\vec \xi-\vec \zeta$; bearing in mind the equations
for $\vec \xi$ and $\vec \zeta$ we have that
 this new unknown satisfies
\begin{equation}\label{ausil}
\frac{\partial\vec \delta}{\partial t} + (-\Delta)^{1+ c }\vec \delta 
   + P[(\vec v\cdot \nabla)\vec \delta -  (\vec \delta \cdot \nabla) \vec v+
     (\vec v\cdot \nabla)\vec \zeta -  (\vec \zeta \cdot \nabla) \vec v]=\vec 0
\end{equation}
Now the quantities $\vec v$ and $\vec \delta$ are linked through
 $\vec \delta=-\vec \zeta+\nabla \times\vec v$.

Our aim is to find  existence and regularity results for $\vec \delta$ in order
to obtain the same results for $\vec \xi$. This requires
$ c \ge\frac 14$.

As in the previous section we look for pathwise results.
\begin{prp}\label{ris-beta}
i) Assume that
\[
\begin{cases}
 c  \ge \frac 14\\
b>\frac 14
\end{cases}
\]
Then, for any $\vec\delta(0)\in L_2$ there exists  
a solution to \eqref{ausil} such
that 
$$
\vec \delta \in   C([0,T];L_2) \cap  L^2(0,T;H^{1+ c })
\qquad\po-a.s.
$$
ii) 
 Assume that
\[
\begin{cases}
 c  \ge \frac 14\\
b>\frac 14\\
2b+ c >\frac 32
\end{cases}
\]
Then, for any $\vec\delta(0)\in H^1$
the solution  given in i)  enjoys also
$$
\vec \delta \in C([0,T];H^1) \cap  L^2(0,T;H^{2+ c }) \qquad\po-a.s.
$$
iii) 
Assume that
\[
\begin{cases}
 c  \ge \frac 14\\
b>\frac 14\\
2b+ c >\frac 52
\end{cases}
\]
Then, for any $\vec\delta(0)\in H^2$
the solution  given in i)  enjoys also
$$
\vec \delta \in C([0,T];H^2) \cap  L^2(0,T;H^{3+ c }) \qquad\po-a.s.
$$
\end{prp}
\begin{proof}
i) First, notice that if $c\ge \frac 14$ and $b>\frac 14$ then  
$2b+ c >\frac 34>\frac 12$. Therefore  Proposition  \ref{pro-zeta} provides  that
for any finite $p$  we have $\vec \zeta \in C([0,T];L_p) $ a.s..

We deal with \eqref{ausil} 
as we did with \eqref{eq-beta}. So
\[
\frac 12 \frac{d}{dt}\|\vec \delta(t)\|_{L_2}^2 
  + \|\vec \delta(t)\|^2_{H^{1+c}} 
=
\langle (\vec \delta \cdot \nabla)\vec  v-
 (\vec v\cdot \nabla)\vec  \zeta
 +(\vec \zeta \cdot \nabla) \vec v, \vec \delta \rangle
\]
From Lemma \ref{st-difficile}
\[
\langle (\vec \delta \cdot \nabla)\vec v,\vec \delta\rangle
\le
\tfrac 16 \|\vec\delta\|^2_{H^{1+c}}
 +C \|\vec v\|_{H^{1+c}}^2  \|\vec \delta\|_{L_2}^2
\]
From \eqref{st2} of Lemma \ref{lemmm}
\[
|\langle(\vec v\cdot \nabla)\vec \zeta , \vec\delta\rangle|
\le  
\tfrac 16 \|\vec \delta\|^2_{H^{1+c}}
  + C \|\vec v \|_{H^1}^2 \|\vec \zeta\|^2_{L_3}
\]
From \eqref{st3} of Lemma \ref{lemmm}
\[
|\langle (\vec \zeta \cdot \nabla)\vec v, \vec \delta\rangle|
\le 
\tfrac 16 \|\vec \delta\|^2_{H^{1+c}}
 + C \|\vec v\|_{H^{1}}^2 \|\vec \zeta\|^2_{L_3}
\]

Summing up, we get
\[
\frac{d}{dt}\|\vec \delta(t)\|_{L_2}^2 
  +  \|\vec\delta(t)\|^2_{H^{1+ c }} 
\le C
\|\vec v(t)\|_{H^{1+ c }}^2 \|\vec \delta(t)\|_{L_2}^2
+ C \|\vec \zeta(t)\|_{L_3}^2 \|\vec v(t)\|_{H^{1+ c }}^2
\]

From Proposition \ref{theta-v}, 
we know that $\vec v \in L^2(0,T;H^{1+ c })$; moreover our
assumption with Proposition \ref{pro-zeta} give $\vec \zeta \in
C([0,]T; H^1_3)$. Then by Gronwall lemma we get
\[
\sup_{0\le t\le T} \|\vec \delta(t)\|_{L_2}^2 <\infty
\]
and integrating in time
\[
 \int_0^T \|\vec \delta(t)\|^2_{H^{1+ c }} \ dt<\infty
\]
The continuity in time is obtained as in
Proposition \ref{pro-beta}.

ii)
 We need a priori estimates and we proceed as in the previous step.
We multiply the l.h.s. of the first equation of \eqref{ausil}
by $-\Delta \vec\delta(t)$ and integrate on $D$. We get
\begin{multline*}
\frac 12 \frac{d}{dt}\|\vec \delta(t)\|_{H^1}^2
+ \|\vec \delta(t)\|_{H^{2+c}}^2
\\=\langle (\vec  v(t)\cdot \nabla)(\vec  \delta(t)
+ \vec \zeta(t)),\Delta \vec\delta(t)\rangle
-\langle ((\vec \delta(t)+\vec \zeta(t))\cdot \nabla) \vec v(t),
   \Delta \vec\delta(t)\rangle
\end{multline*}
We estimate the latter term in the r.h.s. as usual:
\[\begin{split}
|\langle ((\vec \delta+\vec \zeta)\cdot \nabla) \vec v,
   \Delta \vec\delta\rangle|
&\le
\|\vec \delta+\vec \zeta\|_{L_4} \|\nabla \vec v\|_{L_4} \|\Delta \vec\delta\|_{L_2}
\\&
\le C 
\|\vec \delta+\vec \zeta\|_{H^1} \|\nabla \vec v\|_{H^1}
\|\vec\delta\|_{H^2}
\\&
\le C_c
 \|\vec \delta+\vec \zeta\|_{H^1}^2
\|\vec\delta\|_{H^{2+c}}
\\&
\le \frac 14 \|\vec\delta\|_{H^{2+c}}^2
+ C  \|\vec \delta\|_{H^1}^4
+ C \|\vec \zeta\|_{H^1}^4
\end{split}
\]

With this estimate and dealing with  the other trilinear term as in 
the proof of Proposition
\ref{pro-beta} ii), we obtain
\[
\frac{d}{dt}\|\vec \delta(t)\|_{H^1}^2
+ \|\vec \delta(t)\|_{H^{2+c}}^2
\le C \|\vec \delta(t)\|_{H^1}^4 +\|\vec \zeta(t)\|_{H^1}^4
\]
Since $\vec \delta \in L^2(0,T;H^1)$ from the previous step and 
$\vec \zeta \in C([0,T];H^1)$ from Proposition \ref{pro-zeta},
we conclude  as in the proof of Proposition
\ref{pro-beta} ii).

iii)
We multiply the l.h.s. of the first equation of \eqref{ausil}
by $(-\Delta)^2 \vec\delta(t)$ and integrate on $D$. We get
\begin{multline*}
\frac 12 \frac{d}{dt}\|\vec \delta(t)\|_{H^2}^2
+ \|\vec \delta(t)\|_{H^{3+c}}^2
=-\langle (\vec  v(t)\cdot \nabla)(\vec  \delta(t)
+ \vec \zeta(t)),(-\Delta)^2 \vec\delta(t)\rangle
\\+\langle ((\vec \delta(t)+\vec \zeta(t))\cdot \nabla) \vec v(t),
   (-\Delta)^2 \vec\delta(t)\rangle
\end{multline*}
We are left to estimate the latter trilinear term.
First, we use the estimate for the product; by means of 
the Sobolev embeddings $H^2\subset L_\infty$ and $H^1\subset L_4$ we get
\[\begin{split}
\|fg\|_{H^1}
\le \|g\nabla f\|_{L_2}+\|f\nabla g\|_{L_2}
&\le \|\nabla f\|_{L_2} \|g\|_{L_\infty}+\|f\|_{L_4}\|\nabla g\|_{L_4}
\\&\le C \|f\|_{H^1} \|g\|_{H^2}+C\|f\|_{H^1}\|g\|_{H^2}
\end{split}\]
Hence, for the trilinear term we get
\[\begin{split}
\langle ((\vec \delta+\vec \zeta)\cdot \nabla) \vec v,
   (-\Delta)^2 \vec\delta\rangle
&\le  \|((\vec \delta+\vec \zeta)\cdot \nabla) \vec v\|_{H^1}
\|\vec\delta\|_{H^3}
\\&
\le C \|\vec \delta+\vec \zeta\|_{H^1}
\|\nabla \vec v\|_{H^2}\|\vec\delta\|_{H^3}
\\&
\le C_c \|\vec \delta+\vec \zeta\|_{H^1}
\|\vec \delta+\vec \zeta\|_{H^2}\|\vec\delta\|_{H^{3+c}}
\\&
\le \frac 14 \|\vec\delta\|_{H^{3+c}}^2+
C\|\vec \delta+\vec \zeta\|^2_{H^1}\|\vec \delta\|^2_{H^2}
+C\|\vec \delta+\vec \zeta\|^2_{H^1}\|\vec \zeta\|^2_{H^2}
\end{split}
\]
Therefore, keeping in mind the proof of Proposition \ref{pro-beta} iii)
to estimate the other trilinear term, we obtain
\[
\frac{d}{dt}\|\vec \delta(t)\|_{H^2}^2
+ \|\vec \delta(t)\|_{H^{3+c}}^2\le
C\|\vec \delta(t)+\vec \zeta(t)\|^2_{H^1}\|\vec \delta(t)\|^2_{H^2}
+C\|\vec \delta(t)+\vec \zeta(t)\|^2_{H^1}\|\vec \zeta(t)\|^2_{H^2}
\]
Since $\vec \delta \in L^2(0,T;H^2)$ from the previous step and 
$\vec \zeta \in C([0,T];H^2)$ from Proposition \ref{pro-zeta},
we conclude  as in the proof of Proposition
\ref{pro-beta} iii).
\end{proof}

Now we have the result for $\vec \xi=\vec \delta+\vec \zeta$.
\begin{thm}\label{theorem1}
i)
Assume that
\[
\begin{cases}
 c  \ge \frac 14\\
b>\frac 14
\end{cases}
\]
Then, for any $\vec\xi(0)\in L_2$ there exists  
a unique process $\vec \xi$ 
 which is a basic solution to \eqref{vort} such that
\[
 \vec \xi \in 
C([0,T];L_2) \cap   L^2(0,T;L_6)
\]
$\po$-a.s.

Moreover there is continuous dependence on the initial data:
given two
initial data $\vec \xi(0),\vec \xi_\star(0)\in L_2$ we have
\[
 \|\vec \xi(0)-\vec \xi_\star(0)\|_{L_2} \to 0 \Longrightarrow 
 \|\vec\xi-\vec \xi_\star\|_{C([0,T];L_2)} \to 0 
\]
ii)
 Assume that
\[
\begin{cases}
 c  \ge \frac 14\\
b>\frac 14\\
2b+ c >\frac 32
\end{cases}
\]
Then, for any $\vec\xi(0)\in H^1$
the solution  given in i)  enjoys also
$$
\vec \xi \in C([0,T];H^1)  \qquad\po-a.s.
$$
iii) 
Assume that
\[
\begin{cases}
 c  \ge \frac 14\\
b>\frac 14\\
2b+ c >\frac 52
\end{cases}
\]
Then, for any $\vec\xi(0)\in H^2$
the solution  given in i)  enjoys also
$$
\vec \xi \in C([0,T];H^2) \qquad\po-a.s.
$$
\end{thm}
\begin{proof}
i)
If $c\ge \frac 14$ and $b>\frac 14$ then  
$2b+ c >\frac 12$. Therefore  Proposition  \ref{pro-zeta} provides  that
for any finite $p$  we have $\vec \zeta \in C([0,T];L_p) $ a.s..
We merge the results of Proposition \ref{ris-beta} for $\vec \delta$
with those of Proposition \ref{pro-zeta} for $\vec \zeta$  
to get existence of $\vec \xi$ and its regularity. This is the same as
in Theorem \ref{eta-completo}.

As far as continuous dependence on the initial data
 is concerned, we proceed as in the proof of
Theorem \ref{eta-completo}. The additional term does not give any
problem; we estimate it as follows. Set $\vec y = \vec \xi_1-\vec
\xi_2$; then
the system fulfilled by $\vec y$ can be written as
\[
\begin{cases}
\dfrac{\partial\vec y}{\partial t}+( -\Delta)^{1+ c }\vec y 
+P[(\vec v_1\cdot \nabla ) \vec y 
   +\left((\vec v_1-\vec v_2)\cdot \nabla\right)\vec \xi_2
   - (\vec \xi_1\cdot \nabla)(\vec v_1-\vec v_2)
   -(\vec y\cdot \nabla)\vec v_2
]=\vec 0
\\
\nabla \cdot \vec v_1=\nabla \cdot \vec v_2=0
\\
\vec y=\nabla \times (\vec v_1-\vec v_2)
\end{cases}
\]
Therefore,
in the equation fulfilled by $\|\vec y(t)\|_{L_2}^2$,
in addition to the terms appearing in the proof of Theorem
\ref{eta-completo}
we also have
\[
\langle (\vec \xi_1 \cdot \nabla)(\vec v_1-\vec v_2), \vec y\rangle+
\langle (\vec y \cdot \nabla)\vec v_2, \vec y\rangle
\]
We have
\[
\begin{split}
|\langle (\vec \xi_1 \cdot \nabla)(\vec v_1-\vec v_2), \vec y\rangle|
&\le 
\|\vec \xi_1\|_{L_3} \|\nabla (\vec v_1-\vec v_2)\|_{L_2} \|\vec y\|_{L_6}
\\&\le C
\|\vec \xi_1\|_{L_3} \|\vec y\|_{L_2} \|\vec y\|_{H^1}
\\&\le C_ c 
\|\vec \xi_1\|_{L_3} \|\vec y\|_{L_2} \|\vec y\|_{H^{1+c}}
\\&\le 
\frac 16 \|\vec y\|_{H^{1+ c }} ^2 
+ C  \|\vec \xi_1\|^2_{L_3} \|\vec y\|^2_{L_2} 
\end{split}
\]
and
\[
\begin{split}
|\langle (\vec y \cdot \nabla)\vec v_2, \vec y\rangle|
&\le 
\|\vec y\|_{L_2} \|\nabla \vec v_2\|_{L_3} \|\vec y\|_{L_6}
\\&\le C 
\|\vec y\|_{L_2} \|\vec \xi_2\|_{L_3} \|\vec y\|_{H^1}
\\&\le C_ c 
\|\vec y\|_{L_2} \|\vec \xi_2\|_{L_3} \|\vec y\|_{H^{1+c}}
\\&\le 
\frac  16 \|\vec y\|_{H^{1+ c }} ^2 +
C \|\vec \xi_2\|^2_{L_3} \|\vec y\|^2_{L_2} 
\end{split}
\]
Therefore
\[
\frac{d}{dt} \|\vec y(t)\|^2_{L_2}
\le  
C
\big(\|\vec \xi_1(t)\|^2_{L_3}+\|\vec \xi_2(t)\|^2_{L_3}\big) 
\|\vec y(t)\|^2_{L_2} 
\]
By Gronwall lemma, we get continuous dependence on the
initial data. Uniqueness is obtained when $\vec y(0)=\vec 0$
\end{proof}

\section{Equivalence of measures}\label{s:equiv}
Let $\mathcal T : \vec \xi \mapsto \vec v$ be the mapping giving the
solution to \eqref{vel-vort}.

We write system \eqref{vort} as
\begin{equation}\label{xi-breve}
\begin{cases}d \vec \xi+ (-\Delta)^{1+ c }\vec \xi \ dt 
   + P[(\mathcal T \vec \xi\cdot \nabla)\vec \xi] \ dt -  
     P[(\vec \xi \cdot \nabla) \mathcal T \vec \xi] \ dt=(-\Delta)^{-b}d\vec w
\\
\nabla \cdot \vec \xi=0
\end{cases}
\end{equation}
and system \eqref{burg} as 
\begin{equation}\label{eta-breve}
\begin{cases}
 d\vec \eta+ (-\Delta)^{1+ c }\vec \eta \ dt 
+  P[(\mathcal T\vec \eta\cdot \nabla)\vec \eta] \ dt 
 =(-\Delta)^{-b}d\vec w\\
\nabla \cdot \vec \eta=0
\end{cases}
\end{equation}

Denote by $\mathcal L_{\vec\xi}$ and $\mathcal L_{\vec \eta}$
the laws of the processes $\vec \xi$ and $\vec \eta$ respectively, 
when defined on a finite time interval $[0,T]$.
Let $\sigma_T(\vec \eta)$ denote the $\sigma$-algebra generated by
$\{\vec \eta(t)\}_{0\le t\le T}$. 

We recall the main result of  \cite{F}, \cite{Fgir}, in a form
adapted to our context; indeed in those papers it was sufficient to
assume weak existence (without uniqueness) for system \eqref{xi-breve}.
\begin{thm}\label{gener}
Assume \eqref{eta-breve} and \eqref{xi-breve} have a unique basic solution
with the same initial data in $H^2$.
If 
\begin{equation}\label{stimo-eta}
 \po\{ \textstyle\int_0^T
 \|(-\Delta)^b P[(\vec \eta(t)\cdot \nabla )
      \mathcal T\vec \eta(t)]\|_{L_2}^2dt<\infty \}=1,
\end{equation}
\begin{equation}\label{stimo-xi}
 \po\{ \textstyle\int_0^T
 \|(-\Delta)^b P[(\vec \xi(t)\cdot \nabla )
       \mathcal T \vec \xi(t)]\|_{L_2}^2dt<\infty \}=1,
\end{equation}
then  the laws $\mathcal L_{\vec\xi}$ and $\mathcal L_{\vec \eta}$, 
defined as measures on the Borel subsets of $C([0,T];H^2)$,
are equivalent. 

In particular for 
the Radon-Nykodim derivative we have
\begin{equation}
\label{roz}
 \frac{d\mathcal L_{\vec \xi}}{d\mathcal L_{\vec \eta}}(\vec \eta)=
  \mathbb E\left[ e^{\int_0^T 
  \langle  (-\Delta)^b P[(\vec \eta(t)\cdot \nabla )\mathcal T\vec\eta(t)],
      d\vec w(s)\rangle
   -\frac 12 \int_0^T
  \|(-\Delta)^b P[(\vec \eta(t)\cdot \nabla )\mathcal T\vec\eta(t)]\|_{L_2}^2
ds} \big|\sigma_T(\vec \eta)\right] 
\end{equation}
 $\po$-a.s. 

Finally, 
 $\mathcal L_{\vec\xi}$ is unique.
\end{thm}

From this we get our main result.
\begin{thm}
Let 
\[\begin{cases}
 c > \frac 12\\
b=1
\end{cases}
\]
If $\vec \eta(0)=\vec \xi(0)\in H^2$, then the laws 
$\mathcal L_{\vec\xi}$ and $\mathcal L_{\vec \eta}$
are equivalent 
and \eqref{roz} holds.
\end{thm}
\begin{proof}
We use Theorems  \ref{theorem1}, iii); notice that
the conditions on $b$ and $c$  
are fulfilled if $b=1$ and $c>\frac 12$.
We have only to check 
estimates \eqref{stimo-eta} - \eqref{stimo-xi} with
 $b=1$. This follows easily, since $H^2$ is a multiplicative algebra
and $\|\mathcal T\vec \xi\|_{H^3}\le C \|\vec \xi\|_{H^2}$; therefore
\[
\|P[(\vec \xi\cdot \nabla ) \mathcal T\vec \xi]\|_{H^2}\le
C \|\vec \xi\|_{H^2}\|\nabla \mathcal T\vec \xi\|_{H^2}
\le
C \|\vec \xi\|_{H^2}\|\mathcal T\vec \xi\|_{H^3}
\le C\|\vec \xi\|_{H^2}^2
\]
and finally we use that the paths are in $C([0,T];H^2)$.
\end{proof}

We point out that the restriction $c>\frac 12$ cannot be weakened with
this technique using 
\[
\|(-\Delta)^b P[(\vec \zeta\cdot \nabla )
      \mathcal T\vec \zeta]\|_{L_2}^2
\le\| (\vec \zeta\cdot \nabla )
      \mathcal T\vec \zeta\|_{H^{2b}}^2
\le C\|\vec \zeta\|^2_{H^{2b}}
\]
for $b$ large enough.
Indeed, Proposition \ref{pro-zeta}
provides $\zeta \in C([0,T];H^{2b})$ a.s. if $c>\frac 12$.
And the paths of $\vec \xi, \vec \eta$
cannot  have  better behavior than those of $\vec \zeta$.

\bigskip
{\bf Acknowledgments.} 
The author thanks Franco Flandoli for various stimulating
conversations.

\end{document}